\documentclass[12pt]{article}

\usepackage{times}

\usepackage[shortlabels]{enumitem}
\usepackage{enumitem}
\usepackage[utf8]{inputenc}
\usepackage{commath}
\usepackage{amsthm, amscd}
\usepackage{amsmath}
\usepackage{amssymb}
\usepackage{cite}
\usepackage{setspace}
\usepackage{ stmaryrd }
\usepackage{tikz-cd}

\usepackage{tocloft}
\usepackage{sectsty}
\usepackage{xparse}

\newcommand*\tasklabelformat[1]{#1)}

\usepackage{titlesec}

\usepackage{tasks}[newest]
\settasks{
  label = \alph* ,
  label-format = \tasklabelformat ,
  label-width  = 12pt
}

\usepackage{bigints}
\usepackage{comment}
\usepackage{mathtools}
\usepackage{mathrsfs}
\usepackage{fancyhdr}

\allowdisplaybreaks[4]

\numberwithin{equation}{section}
\usepackage{etoolbox}
\patchcmd{\thebibliography}
  {\settowidth}
  {\setlength{\itemsep}{0pt plus -10pt}\settowidth}
  {}{}
\apptocmd{\thebibliography}
  {
  }
  {}{}
\makeatletter
\newtheorem*{rep@theorem}{\rep@title}
\newcommand{\newreptheorem}[2]{%
\newenvironment{rep#1}[1]{%
 \def\rep@title{#2 \ref{##1}}%
 \begin{rep@theorem}}%
 {\end{rep@theorem}}}
\makeatother

\theoremstyle{theorem}

\newreptheorem{theorem}{Theorem}
\newtheorem{thm}{Theorem}[section]
\newtheorem*{thm*}{Theorem}
\theoremstyle{definition}
\newtheorem{prop}[thm]{Proposition}
\newtheorem*{prop*}{Proposition}
\newtheorem{defn}[thm]{Definition}
\newtheorem{lem}[thm]{Lemma}
\newtheorem{cor}[thm]{Corollary}
\newtheorem*{cor*}{Corollary}
\theoremstyle{remark}

\newtheorem{rem}[thm]{Remark}

\title{\vspace*{-1.5cm} Bernstein-Markov measures and Toeplitz theory}

\author
{Siarhei Finski
}

\date{}

\usepackage[%
    left=1in,%
    right=1in,%
    top=1.1in,%
    bottom=0.8in,%
    paperheight=11in,%
    paperwidth=8.5in%
]{geometry}

\newcommand{\imun} {\sqrt{-1}}

\newcommand{\comp}{\mathbb{C}}
\newcommand{\real}{\mathbb{R}}

\newcommand{\nat}{\mathbb{N}}

\newcommand{\enmr}[1]{\text{End}{(#1)}}

\newcommand{\ccal}{\mathscr{C}}

\makeatletter
\DeclareFontFamily{OMX}{MnSymbolE}{}
\DeclareSymbolFont{MnLargeSymbols}{OMX}{MnSymbolE}{m}{n}
\SetSymbolFont{MnLargeSymbols}{bold}{OMX}{MnSymbolE}{b}{n}
\DeclareFontShape{OMX}{MnSymbolE}{m}{n}{
    <-6>  MnSymbolE5
   <6-7>  MnSymbolE6
   <7-8>  MnSymbolE7
   <8-9>  MnSymbolE8
   <9-10> MnSymbolE9
  <10-12> MnSymbolE10
  <12->   MnSymbolE12
}{}
\DeclareFontShape{OMX}{MnSymbolE}{b}{n}{
    <-6>  MnSymbolE-Bold5
   <6-7>  MnSymbolE-Bold6
   <7-8>  MnSymbolE-Bold7
   <8-9>  MnSymbolE-Bold8
   <9-10> MnSymbolE-Bold9
  <10-12> MnSymbolE-Bold10
  <12->   MnSymbolE-Bold12
}{}

\let\llangle\@undefined
\let\rrangle\@undefined
\DeclareMathDelimiter{\llangle}{\mathopen}%
                     {MnLargeSymbols}{'164}{MnLargeSymbols}{'164}
\DeclareMathDelimiter{\rrangle}{\mathclose}%
                     {MnLargeSymbols}{'171}{MnLargeSymbols}{'171}
\makeatother

\newenvironment{sciabstract}{}

\cftsetindents{section}{0em}{1.7em}

\sectionfont{\large}

\titlespacing*{\section}
{0pt}{5pt}{5pt}

\begin{document}

\maketitle 

\vspace*{-0.7cm}

\vspace*{0.3cm}

\begin{sciabstract}
  \textbf{Abstract.} 
	We prove that Toeplitz operators associated with a Bernstein-Markov measure on a compact complex manifold endowed with a big line bundle form an algebra under composition.
	As an application, we derive a Szegő-type spectral equidistribution result for this class of operators.
	A key component of our approach is the off-diagonal asymptotic analysis of the Bergman kernel, also known as the Christoffel-Darboux kernel.
\end{sciabstract}

\pagestyle{fancy}
\lhead{}
\chead{Bernstein-Markov measures and Toeplitz theory}
\rhead{\thepage}
\cfoot{}


\newcommand{\Addresses}{{
  \bigskip
  \footnotesize
  \noindent \textsc{Siarhei Finski, CNRS-CMLS, École Polytechnique F-91128 Palaiseau Cedex, France.}\par\nopagebreak
  \noindent  \textit{E-mail }: \texttt{finski.siarhei@gmail.com}.
}} 

\vspace*{0.25cm}

\par\noindent\rule{1.25em}{0.4pt} \textbf{Table of contents} \hrulefill

\vspace*{-1.5cm}

\tableofcontents

\vspace*{-0.2cm}

\noindent \hrulefill


\section{Introduction}\label{sect_intro}
	The primary aim of this paper is to study Toeplitz operators associated with Bernstein-Markov measures on compact complex manifolds endowed with big line bundles.  
	As a byproduct of independent interest, we study the off-diagonal asymptotics of the corresponding Bergman kernel.
	\par 
	Throughout the whole article, we fix a compact complex manifold $X$, $\dim_{\comp} X = n$, and a big line bundle $L$ over $X$, i.e. such that for some $c > 0$, we have $\dim H^0(X, L^{\otimes k}) \geq c k^n$, for $k \in \nat$ big enough.  
	We fix a continuous Hermitian metric $h^L$ on $L$ and a positive Borel measure $\mu$ supported on a compact subset $K \subset X$.
	We assume that $\mu$ is \textit{non-pluripolar}, i.e. it does not charge pluripolar sets -- recall that a subset is pluripolar if it is contained in the $\{-\infty\}$-locus of some plurisubharmonic (\textit{psh}) function.
	\par 
	We denote by ${\textrm{Hilb}}_k(h^L, \mu)$ the positive semi-definite form on $H^0(X, L^{\otimes k})$ defined for arbitrary $s_1, s_2 \in H^0(X, L^{\otimes k})$ as follows
	\begin{equation}\label{eq_defn_l2}
			\langle s_1, s_2 \rangle_{{\textrm{Hilb}}_k(h^L, \mu)} = \int_X \langle s_1(x), s_2(x) \rangle_{(h^L)^{k}} \cdot d \mu(x).
	\end{equation}
	Remark that since $\mu$ is non-pluripolar, the above form is positive definite.
	\par
	Recall that $\mu$ is said to be \textit{Bernstein-Markov} with respect to $(K, h^L)$, if for each $\epsilon > 0$, there is $C > 0$, such that for any $k \in \nat$, we have
	\begin{equation}\label{eq_bm_clas}
		{\textrm{Ban}}_k^{\infty}(K, h^L)
		\leq
		C
		\cdot
		\exp(\epsilon k)
		\cdot
		{\textrm{Hilb}}_k(h^L, \mu),
	\end{equation}
	where ${\textrm{Ban}}_k^{\infty}(K, h^L)$ stands for the $L^{\infty}(K)$-norm on $H^0(X, L^{\otimes k})$ induced by $h^L$. 
	For a survey on the Bernstein-Markov property, we recommend \cite{BernsteinMarkovSurvey}.  
	Here, we only mention that the Lebesgue measure on a smoothly bounded domain in $X$, or on a totally real compact submanifold of real dimension $n$, is Bernstein-Markov, see \cite[Corollary 1.8]{BerBoucNys}, \cite[Proposition 3.6]{BernsteinMarkovSurvey} and \cite[Theorem 1.3]{MarinescuVu}, and also the recent survey \cite{MarinescuVuSurvey}.
	\par
	\begin{center}
		From now on we assume that $\mu$ is Bernstein-Markov with respect to $(K, h^L)$.
	\end{center}
	\par 
	Now, we fix a non-pluripolar subset $E \subset X$. 
	Following Siciak \cite{SiciakExtremal}, Guedj-Zeriahi \cite{GuedZeriGeomAnal}, we define the \textit{psh envelope} $h^L_{E}$ associated with $E$ as
	\begin{equation}\label{defn_env}
		h^L_E
		=
		\inf \Big\{
			h^L_0 \text{ with psh potential }: h^L_0 \geq h^L \text{ over } E
		\Big\}.
	\end{equation}
	When $L$ is ample, the non-pluripolarity of $E$ assures that the metric $h^L_E$ has a bounded potential, cf. \cite[Theorem 9.17]{GuedZeriGeomAnal}.
	When $L$ is merely big, the potential of $h^L_E$ might have singularities.
	In either case, the lower semi-continuous regularization $h^L_{E *}$ of $h^L_E$ is a metric with a psh potential, and it has the minimal singularities, cf. \cite[\S 1.2]{BermanBouckBalls} for the necessary definitions.
	\par 
	Recall that the equilibrium measure associated with $(K, h^L)$ is defined as
	\begin{equation}\label{eq_equil_meas}
		\mu_{\mathrm{eq}}(K, h^L) := \frac{1}{{\rm{vol}}(L)} c_1(L, h^L_{K *})^n.
	\end{equation}
	Above, ${\rm vol}(L)$ denotes the volume of the line bundle $L$, defined by
	\begin{equation}
		{\rm{vol}}(L) := \lim_{k \to \infty} \frac{n!}{k^n} \dim H^0(X, L^{\otimes k}),
	\end{equation}
	a limit that exists by Fujita’s theorem \cite{Fujita} and is strictly positive under our assumption that $L$ is big. 
	The positive $(1,1)$-current $c_1(L, h^L_{K *})$ is defined as $ c_1(L, h^L_{K *}) := c_1(L, h^L_0) + \frac{\imun}{2 \pi} \partial \bar{\partial} \phi_K$, where $h^L_{K *} = h^L_0 \cdot \exp(-\phi_K)$ and $h^L_0$ is an arbitrary smooth metric on $L$. 
	It is generally non-smooth by the above discussion.
	When $L$ is ample, the top wedge power $c_1(L, h^L_{K*})^n$ can be defined in the sense of Bedford-Taylor \cite{BedfordTaylor}, since $h^L_{K*}$ admits a bounded potential. 
	However, as we assume $L$ is merely big, the wedge product $c_1(L, h^L_{K*})^n$ appearing in (\ref{eq_equil_meas}) is interpreted via the non-pluripolar product developed by Boucksom-Eyssidieux-Guedj-Zeriahi \cite{BEGZ}, which generalizes the Bedford-Taylor construction. 
	In this case, $\mu_{\mathrm{eq}}(K, h^L)$ is a probability measure on $X$ supported in $K$, see \cite[Proposition 1.10]{BermanBouckBalls}.
	\par 
	From now on, we assume that $k$ is large enough so that $n_k > 0$, where $n_k := \dim H^0(X, L^{\otimes k})$.
	We define a sequence of measures on $X \times X$ as follows
	\begin{equation}\label{eq_mu_begmm}
		\mu_k^{\rm{Berg}} := \frac{1}{n_k} |B_k(x, y)|^{2}_{(h^L)^k} \cdot d \mu(x) \cdot d \mu(y),
	\end{equation}
	where $B_k(x, y) \in L^k_x \otimes (L^k_y)^*$ is the Bergman (or Christoffel-Darboux) kernel, defined as
	\begin{equation}\label{eq_bergm_kern}
		B_k(x, y) := \sum_{i = 1}^{n_k} s_i(x) \cdot s_i(y)^*,
	\end{equation}
	where $s_i$, $i = 1, \ldots, n_k$ is an orthonormal basis of $(H^0(X, L^{\otimes k}), {\textrm{Hilb}}_k(h^L, \mu))$. 
	The reader will check that (\ref{eq_bergm_kern}) doesn't depend on the choice of the basis, and that $\mu_k^{\rm{Berg}}$, $k \in \nat$, are probability measures on $X \times X$, cf. (\ref{eq_push_forw}).
	The following result is one of the main results of this paper.
	\par 
	\begin{thm}\label{thm_off_diag}
		The measures $\mu_k^{\rm{Berg}}$ converge weakly, as $k \to \infty$, to the measure $\Delta_* \mu_{eq}(K, h^L)$, where $\Delta: X \to X \times X$ is the diagonal embedding.
	\end{thm}
	\begin{rem}
		a) As we shall explain in Section \ref{sect_off_diag}, Theorem \ref{thm_off_diag} refines the convergence of Bergman measures, due to Berman-Boucksom-Witt Nyström \cite{BerBoucNys}, cf. Theorem \ref{thm_diag}.
		We, however, stress out that our proof depends on \cite{BerBoucNys}, and the only novel aspect of Theorem \ref{thm_off_diag} is the proof of the asymptotic concentration of the mass of $\mu_k^{\mathrm{Berg}}$ along the diagonal.
		\par 
		b) We refer to Section \ref{sect_off_diag} for the discussion on the sharpness of Theorem \ref{thm_off_diag}.
		\par 
		c) When $\mu$ is a volume form and $h^L$ is smooth, Theorem \ref{thm_off_diag} was established by Berman \cite[Theorem 1.6]{BermanEnvProj} by different techniques. 
		The validity of the general case of Theorem \ref{thm_off_diag} was left open in Zelditch \cite[after Theorem 2.5]{ZelditchErgodSurvey}.
	\end{rem}
	\par 
	To state our application of Theorem \ref{thm_off_diag}, for a fixed $f \in \ccal^0(X)$ and $k \in \nat^*$, we define $T_k(f) \in {\enmr{H^0(X, L^{\otimes k})}}$ as $T_k(f) := B_k \circ M_k(f)$, where $B_k : \ccal^0(X, L^{\otimes k}) \to H^0(X, L^{\otimes k})$ is the orthogonal (Bergman) projection to $H^0(X, L^{\otimes k})$ with respect to (\ref{eq_defn_l2}), and $M_k(f) : H^0(X, L^{\otimes k}) \to \ccal^0 (X, L^{\otimes k})$ is the multiplication map by $f$.
	\par 
	By considering the products $\langle T_k(f) s_1, s_2 \rangle_{{\textrm{Hilb}}_k(h^L, \mu)}$ for arbitrary $s_1, s_2 \in H^0(X, L^{\otimes k})$, we see that $T_k(f)$ depends solely on the restriction of $f$ to $K$.
	As by Tietze-Urysohn-Brouwer extension theorem, any function from $\ccal^0(K)$ admits an extension to $\ccal^0(X)$, we can thus extend the definition of $T_k(f)$ for any $f \in \ccal^0(K)$.
	\begin{defn}\label{defn_toepl_sch}
		A sequence of operators $T_k \in {\enmr{H^0(X, L^{\otimes k})}}$, $k \in \nat$, is called a \textit{Toeplitz operator} if there is $C > 0$, such that for any $k \in \nat$, $\| T_k \| \leq C$, where $\| \cdot \|$ is the operator norm, and there is $f \in \ccal^0(K)$, called the symbol of $\{T_k\}_{k = 0}^{+ \infty}$, so that for any $\epsilon > 0$, $p \in [1, +\infty[$, there is $k_0 \in \nat$, such that for every $k \geq k_0$, we have
		\begin{equation}\label{eq_toepl_schatten}
			\big \| T_k -  T_k(f) \big \|_p \leq \epsilon,
		\end{equation}
		where $\| \cdot \|_p$ is the $p$-Schatten norm, defined for an operator $A \in {\enmr{V}}$, of a finitely-dimensional Hermitian vector space $(V, H)$ as $\| A \|_p = (\frac{1}{\dim V} {\rm{Tr}}[|A|^p])^{\frac{1}{p}}$, $|A| := (A A^*)^{\frac{1}{2}}$.
	\end{defn}
	\begin{rem}
		a) The standard definition of Toeplitz operators uses the operator norm rather than the $p$-Schatten norm in (\ref{eq_toepl_schatten}), and the measure $\mu$ is taken to be a volume form, cf. \cite{MaHol}. 
		The variant from Definition \ref{defn_toepl_sch} is a mix of two definitions previously introduced by the author in \cite{FinSubmToepl}, \cite{FinEigToepl} under the names “Toeplitz operator of Schatten class" and “Generalized Toeplitz operator". 
		Some similar definitions appeared before in Grenander-Szeg{\"o} \cite[\S 7.4]{GrenanSzego}.
		As we work here exclusively with this version of Toeplitz operators, we simplify our terminology accordingly.
		\par 
		b) 
		Standard estimates, cf. (\ref{eq_schatted_bnds}), imply that it suffices to verify (\ref{eq_toepl_schatten}) for a single $p \in [1, +\infty[$.
	\end{rem}
	\par 
	As we shall explain in Proposition \ref{prop_symb_well_def}, if a Toeplitz operator with symbol $f \in \ccal^0(K)$ is also a Toeplitz operator with symbol $g \in \ccal^0(K)$, then the restrictions of $f$ and $g$ over 
	\begin{equation}\label{eq_kprim}
		K' := {\rm{supp}}(\mu_{\mathrm{eq}}(K, h^L)) \subset K
	\end{equation}
	coincide.
	Likewise, if the restrictions of $f$ and $g$ over $K'$ coincide, then any Toeplitz operator with symbol $f$ is also a Toeplitz operator with symbol $g$.
	Henceforth, the \textit{symbol map} sending a Toeplitz operator with symbol $f$ to the restriction of $f$ on $K'$ is well-defined.	
	\par 
	When $\mu$ is a volume form on $X$, $L$ is an ample line bundle and $h^L$ is a smooth metric with a strictly positive curvature, Toeplitz operators have found numerous applications in complex and algebraic geometry, see \cite{BerndtProb}, \cite{MaZhangSuperconnBKPubl}, \cite{FinSubmToepl}, \cite{FinHNII}. 
	It is impossible to provide a comprehensive survey of all the results in this direction, so we only mention a few works: \cite{BoutGuillSpecToepl}, \cite{BordMeinSchli}, \cite{CharlesToepl}, \cite{MaMarinGeneralizedBK}, \cite{MaMarBTKah}, and \cite{MaHol} for classical results, and \cite{AnconaFloch}, \cite{HerrmanHsiaoMarinShen}, \cite{DrewLiuMar2}, \cite{GalassoHsiao}, \cite{DelepSpec1}, \cite{OodOffDiag} for more recent developments.
	\par 
	As we explain in Section \ref{sect_mat}, taking $X := \mathbb{P}^1$, $L := \mathscr{O}(1)$, and letting $\mu$ be the Lebesgue measure on a great circle, Definition \ref{defn_toepl_sch} recovers the classical notion of Toeplitz matrices. 
	These were first introduced by Toeplitz in \cite{Toeplitz} and later studied systematically by Szeg\H{o} in \cite{SzegoFST}, \cite{SzegoSND}.  
	Many of the results discussed in the complex geometric setting above have analogues in the classical theory of Toeplitz matrices -- often established much earlier.  
	For a comprehensive survey of this classical theory, we refer the reader to the monographs by Grenander-Szeg{\"o} \cite{GrenanSzego} and Nikolski \cite{NikolskiBook}.
	\par 
	We unify these parallel theories, by showing that the key results continue to hold in much greater generality.
	More specifically, our second main result of this paper goes as follows.
	\begin{thm}\label{thm_alg}
		The space of Toeplitz operators forms an algebra under composition, and the symbol map is an algebra morphism when $\ccal^0(K')$ is endowed with the standard multiplication.
		In particular for any $f, g \in \ccal^0(K)$, $p \in [1, +\infty[$, we have 
		\begin{equation}\label{eq_toepl_comp_as}
			\lim_{k \to \infty} \| T_k(f) \circ T_k(g) - T_k(f \cdot g) \|_p = 0.
		\end{equation}
	\end{thm}
	\begin{rem}\label{rem_alg}
	 	Bordemann–Meinrenken–Schlichenmaier in \cite{BordMeinSchli} showed that when $\mu$ is a volume form on $X$, $L$ is an ample line bundle and $h^L$ is a smooth metric with a strictly positive curvature, the analogous statement holds if the $p$-Schatten norm in Definition \ref{defn_toepl_sch} is replaced by the operator norm. 
	 	The author in \cite{FinSubmToepl} proved that under the same assumptions, Theorem \ref{thm_alg} remains valid when the space $\mathcal{C}^0(X)$ in Definition \ref{defn_toepl_sch} is replaced by $L^\infty(X)$.
	 	It remains open if in our setting it is possible to get the analogues of the above results. 
	\end{rem}
	It would be interesting to identify natural conditions on $\mu$ that allow for a refinement of (\ref{eq_toepl_comp_as}) to include lower-order terms in $k$, and ultimately to analyze the commutator $\frac{1}{k} [T_k(f), T_k(g)]$ and its connections to symplectic geometry, particularly in light of the foundational works \cite{BordMeinSchli} and Ma-Marinescu \cite{MaMarBTKah}.
	\par 
	As an application of Theorem \ref{thm_alg}, we get the following equidistribution result for the spectral measures of Toeplitz operators.
	\begin{thm}\label{thm_distr}
		For any Toeplitz operator $\{ T_k, k \in \nat \}$ with symbol $f \in \ccal^0(K)$, and any continuous $g: \real \to \real$, we have
		\begin{equation}
			\lim_{k \to \infty} \frac{1}{n_k} \sum_{\lambda \in {\rm{Spec}} (T_k)} g(\lambda)
			=
			\int_X g(f(x)) d \mu_{eq}(K, h^L)(x).
		\end{equation}
		In particular, the asymptotic shape of the spectrum of a Toeplitz operator depends solely on the image of the symbol map.
	\end{thm}
	\begin{rem}
		a) When $\{ T_k, k \in \nat \}$ forms a Toeplitz matrix, the result recovers Szeg\"o first limit theorem \cite{SzegoFST}, see Section \ref{sect_mat} for details.
		Under the same assumptions as in Remark \ref{rem_alg}, Theorem \ref{thm_distr} was established by Boutet de Monvel-Guillemin \cite{BoutGuillSpecToepl}, see also Ma-Marinescu \cite{MaMarinGeneralizedBK}, \cite{MaMarBTKah}, \cite[\S 7]{MaHol} and \cite[Appendix A]{FinMaVol} for a more direct proof relying on Bergman kernels.
		\par 
		b) When the line bundle $L$ is endowed with a smooth Hermitian metric with semi-positive curvature, and $\mu$ is a volume form, Theorem \ref{thm_distr} was established by Berman \cite[Theorem 2.7]{BermanSuperToepl}.
	\end{rem}
	\par 
	This article is organized as follows.  
	In Section \ref{sect_off_diag}, we prove Theorem \ref{thm_off_diag} using the diagonal asymptotics of the Bergman measure.  
	Section \ref{sect_toepl} is devoted to the study of Toeplitz operators, where we establish Theorems \ref{thm_alg} and \ref{thm_distr}.  
	In Section \ref{sect_mat}, we demonstrate how the main results of this paper allow us to recover classical results on Toeplitz matrices and orthogonal polynomials.
	Finally, in Section \ref{sect_sing}, we discuss an application to Bergman kernels on singular spaces.
	\par
	\textbf{Acknowledgement.}
	The author acknowledges the support of CNRS, École Polytechnique and the partial support of ANR projects QCM (ANR-23-CE40-0021-01), AdAnAr (ANR-24-CE40-6184) and STENTOR (ANR-24-CE40-5905-01).

	\section{Off-diagonal estimates on the Bergman measure}\label{sect_off_diag}
	
	The main goal of this section is to establish that the sequence of measures $\mu_k^{\rm{Berg}}$ does not asymptotically place mass away from the diagonal.
	As a consequence, we establish Theorem \ref{thm_off_diag}.
	More specifically, the main result of the section goes as follows.
	\begin{thm}\label{thm_nomass}
		Under the same notations and assumptions as in Theorem \ref{thm_off_diag}, for any compact subsets $K_1, K_2 \subset X$ verifying $K_1 \cap K_2 = \emptyset$, we have
		\begin{equation}
			\lim_{k \to \infty} \int_{K_1 \times K_2} \mu_k^{\rm{Berg}} = 0.
 		\end{equation}
	\end{thm}
	\par 
	Before turning to the proof of Theorem \ref{thm_nomass}, we discuss the extent to which the result is sharp.
	Let us consider the case where $L$ is an ample line bundle, $h^L$ is a smooth Hermitian metric with strictly positive curvature, and $\mu$ is a volume form.  
	The off-diagonal expansion of the Bergman kernel in this context has been thoroughly investigated by Christ \cite{ChristOffDiagDim1} and by Ma-Marinescu \cite{MaMarOffDiag}, the latter also providing a comprehensive survey of such estimates.
	Their results imply that there exists a constant $c > 0$, depending on $K_1$ and $K_2$, such that $\int_{K_1 \times K_2} \mu_k^{\mathrm{Berg}} \leq \exp(-c \sqrt{k})$, and a similar bound holds for the Bergman kernel away from the diagonal.
	\par 
	We underline that the above exponential estimate hinges on the spectral gap property of the Kodaira Laplacian, a property which fails even for smooth semi-positive metrics $h^L$ on non-ample $L$, see Donnelly \cite{DonnellyGap}. 
	In the broader setting of non-pluripolar measures considered here, the spectral approach does not seem to be applicable.
	Moreover, the following example suggests that one cannot substantially improve upon Theorem \ref{thm_nomass}.
	\par 
	\textbf{Example 1.}
	We embed $\mathbb{S}^1$ in $X := \mathbb{P}^1$ as one of the great circles (for concreteness given by $\theta \mapsto [1: \exp(i \theta)] \in \mathbb{P}^1$, $\theta \in [0, 2 \pi[$, where $[1 : z] \in \mathbb{P}^1$, $z \in \comp$, is a standard affine chart) and denote by $\mu$ the Lebesgue measure on $\mathbb{S}^1$, viewed as a measure on $X$.
	Remark that as $\mathbb{S}^1$ is totally real, the measure $\mu$ is non-pluripolar, see \cite{SadullaevReal}, cf. \cite[Exercise 4.39.5]{GuedjZeriahBook}. 
	We then take $L := \mathscr{O}(1)$ endowed with the Fubini-Study metric $h^{FS}$.
	By the results recalled in the Introduction, the measure $\mu$ is Bernstein-Markov for $(\mathbb{S}^1, h^{FS})$. 
	\par 
	An easy calculation shows that with respect to the $L^2$-product associated with the above $\mu$, the standard basis of monomials of $H^0(X, L^{\otimes k})$ given by $z^i$, $i = 0, \ldots, k - 1$, in the already mentioned affine chart, is orthonormal.
	Hence, we conclude that for any $x, y \in \mathbb{S}^1$, we have
	\begin{equation}
		B_k(x, y) = \sum_{i = 0}^{k - 1} x^i \cdot \overline{y}^i.
	\end{equation}
	Using this formula, the reader will first observe that $|B_k(x, x)| = k$ for all $x \in \mathbb{S}^1$, $k \in \mathbb{N}$. 
	Also, it immediately follows that for odd $k$ one has $|B_k(x, -x)| = 1$. 
	As a consequence, no exponential decay occurs.  
	Moreover, one checks that for non-intersecting non-empty intervals $K_1$ and $K_2$, we have $\int_{K_1 \times K_2} \mu_k^{\mathrm{Berg}} \sim c(K_1, K_2) \cdot k^{-1}$ for some $c(K_1, K_2) > 0$, which rules out any possibility of exponential decay of the mass.
	\par 
	Our next example will indicate that the bigness assumption on $L$ is crucial for Theorem \ref{thm_nomass}.
	\par 
	\textbf{Example 2.}
	Consider $X := X_1 \times X_2$, where $X_1$, $X_2$ are connected projective complex manifolds of positive dimensions, and let $L := \pi_2^* L_2$, where $L_2$ is an ample line bundle over $X_2$, and $\pi_2 : X \to X_2$ is the natural projection.
	We endow $L_2$ with a smooth Hermitian metric $h^L_2$ with positive curvature, and let $h^L := \pi_2^* h^L_2$.
	Endow $X$ with a volume form $\mu$, given by the wedge product of the pull-backs of volume forms $\mu_1$ on $X_1$ and $\mu_2$ on $X_2$.
	Then the Bergman kernel $B_k(x, y)$, $x , y \in X$, associated with $\mu$ and $h^L$ will relate to the Bergman kernel $B_{k, 2}(x, y)$, $x , y \in X_2$, associated with $\mu_2$ and $h^L_2$ as follows $B_k(x, y) = B_{k, 2}(\pi_2(x), \pi_2(y)) / \int_{X_1} d \mu_1$.
	If $K_{1, 1}, K_{1, 2} \subset X_1$ are two non-intersecting compact subsets of positive volume, then for the non-intersecting compact subsets $K_1 := K_{1, 1} \times X_2, K_2 := K_{1, 2} \times X_2$ of $X$, we get 
	\begin{multline}
		\int_{K_1 \times K_2} |B_k(x, y)|^2_{(h^L)^k} \cdot d \mu(x) \cdot d \mu (y)
		\\
		=
	 	\frac{\int_{K_{1, 1}} d \mu_1}{\int_{X_1} d \mu_1}
	 	\cdot
	 	\frac{\int_{K_{1, 2}} d \mu_1}{\int_{X_1} d \mu_1}
	 	\cdot
	 	\int_{X_2 \times X_2} |B_k(x, y)|^2_{(h^L_2)^k} \cdot d \mu_2 (x) \cdot d \mu_2 (y).
	\end{multline}
	And since $\int_{X_2 \times X_2} |B_k(x, y)|^2_{(h^L_2)^k} d \mu_2 (x) \cdot d \mu_2 (y) = \dim H^0(X_2, L_2^{\otimes k})$, see (\ref{eq_push_forw}), the above identity shows that Theorem \ref{thm_nomass} doesn't hold in the above setting.
	\par 
	In our final example we point out that Theorem \ref{thm_nomass} cannot be improved by a pointwise estimate on the off-diagonal Bergman kernel.
\par 
	\textbf{Example 3.}
	Consider a projective complex manifold $X_1$ of positive dimension $n$, endowed with an ample line bundle $L_1$.
	We endow $L_1$ with a smooth Hermitian metric $h^L_1$ of positive curvature.
	Consider a blow-up $\pi : X \to X_1$ at a point $x_1 \in X_1$, and let $L := \pi^* L_1$, $h^L := \pi^* h^L_1$.
	Endow $X_1$ with a volume form $\mu_1$ and define $\mu := \pi^* \mu_1$.
	Then the Bergman kernel $B_k(x, y)$, $x , y \in X$, associated with $\mu$ and $h^L$ will relate to the Bergman kernel $B_{k, 1}(x, y)$, $x , y \in X_1$, associated with $\mu_1$ and $h^L_1$ as $B_k(x, y) = B_{k, 1}(\pi(x), \pi(y))$.
	We see in particular that for any $x, x' \in X$ so that $\pi(x), \pi(x') = x_1$, we have $|B_k(x, x')|_{(h^L)^k} = B_{k, 1}(x_1, x_1)$.
	But by the result of Tian \cite{TianBerg}, cf. discussion after Theorem \ref{thm_diag}, $B_{k, 1}(x_1, x_1) \sim k^n$, as $k \to \infty$, and so in Theorem \ref{thm_nomass} one cannot replace the mass of the measure by the pointwise bound on the Bergman kernel.
	\par 
	Now, quite surprisingly, our main ingredient in the proof of Theorem \ref{thm_nomass} is based on the study of the Bergman kernel \textit{on the diagonal}.
	We rely more specifically on one of the main results of Berman-Boucksom-Witt Nyström \cite{BerBoucNys} (see also Bloom-Levenberg \cite{BloomLeven1,BloomLeven2} for earlier results in this direction, and Bayraktar-Bloom-Levenberg \cite{BayrBloomLeven} for further developments), which also plays a key role in the proof of Theorem \ref{thm_distr}.
	\begin{thm}[{\cite[Theorem B]{BerBoucNys}}]\label{thm_diag}
		Under the same notations and assumptions as in Theorem \ref{thm_off_diag}, the sequence of measures $\frac{1}{n_k} |B_k(x, x)| \cdot d \mu(x)$ on $X$ converges weakly, as $k \to \infty$, to the equilibrium measure $\mu_{eq}(K, h^L)$.
	\end{thm}
	\par 
	Remark that when $\mu$ is a volume form on $X$, $L$ is an ample line bundle and $h^L$ is a smooth metric with a strictly positive curvature, much more precise asymptotic results on the Bergman kernel can be obtained.  
	In this setting, it is known that $|B_k(x, x)|$ admits a full asymptotic expansion in powers of $k$, and this expansion depends smoothly on all relevant parameters -- namely, the point $x \in X$, the metric on the manifold, the volume form, the complex structure, and so on.  
	This result was established in various degrees of generality by Tian \cite{TianBerg}, Catlin \cite{Caltin}, Bouche \cite{Bouche}, Zelditch \cite{ZeldBerg}, Dai-Liu-Ma \cite{DaiLiuMa}, Ma-Marinescu \cite{MaHol}, and Ma-Zhang \cite{MaZhangSuperconn}.
	\par 
	Note, however, that such a strong pointwise asymptotic expansion is highly sensitive to the underlying assumptions.  
	For instance, already in the case of a Riemann surface endowed with an ample line bundle with semi-positive and not strictly positive curvature, the nature of the expansion changes significantly and depends on the order of vanishing of the curvature, as demonstrated by Marinescu-Savale \cite{MarinSaval}.
	\par 
	We also note that, under stronger assumptions, one can study the speed of convergence in Theorem \ref{thm_diag}, as was done by Dinh-Ma-Nguyên \cite[Theorem 2.11]{DinhMaNgu}. 
	It would be interesting to investigate whether similar estimates can be obtained in the setting of Theorem \ref{thm_off_diag}.
	\par 
	Let us finally explain the relation between Theorems \ref{thm_off_diag} and \ref{thm_diag}.
	We denote by $\pi: X \times X \to X$ one of the two projections. 
	Then an easy verification (see (\ref{eq_bk_id})) shows
	\begin{equation}\label{eq_push_forw}
		\pi_* \mu_k^{\rm{Berg}}
		=
		\frac{1}{n_k} |B_k(x, x)| \cdot d \mu(x).
	\end{equation}
	Since pushforwards under continuous maps preserve weak convergence, we observe that Theorem \ref{thm_off_diag} constitutes a refinement of Theorem \ref{thm_diag}.
	\par 
	\begin{proof}[Proof of Theorem \ref{thm_nomass}]
		Assume for contradiction that there exists $\epsilon > 0$ such that, up to passing to a subsequence, for all $k \in \nat$ big enough, we have 
		\begin{equation}\label{eq_contr}
    		\int_{K_1 \times K_2} \mu_k^{\mathrm{Berg}} \geq \epsilon.
		\end{equation}
		From Theorem \ref{thm_diag}, (\ref{eq_push_forw}) and (\ref{eq_contr}), we then immediately obtain
		\begin{equation}\label{eq_non_empt}
			\int_{K_1} \mu_{\mathrm{eq}}(K, h^L)
			\geq
			\epsilon.
		\end{equation}
 		\par 
		By a version of Urysohn's lemma, we pick two continuous functions $f, g : X \to [0, 1]$, with non-intersecting support, and such that $f|_{K_1} = 1$, $g|_{K_2} = 1$.
		We consider the measure 
		\begin{equation}
			\mu' := \mu \cdot \exp(- g).
		\end{equation}
		It follows directly from the definition that $\mu'$ is a Bernstein-Markov measure for $(K, h^L)$, since $\mu$ is, and so Theorem \ref{thm_diag} holds for $\mu := \mu'$.
		Hence, if we denote by $B_k'(x, x)$, $x \in X$, $k \in \nat$, the Bergman kernel associated with $\mu'$ and $h^L$, by Theorem \ref{thm_diag}, we conclude that 
		\begin{equation}\label{eq_asympt_bk_diag}
		\begin{aligned}
			&
			\lim_{k \to \infty}
			\frac{1}{n_k}
			\int_{X} f(x) \cdot B_k'(x, x) \cdot d \mu'(x)
			=
			\int_{X} f \cdot d \mu_{\mathrm{eq}}(K, h^L), 
			\\
			&
			\lim_{k \to \infty}
			\frac{1}{n_k}
			\int_{X} f(x) \cdot B_k(x, x) \cdot d \mu(x)
			=
			\int_{X} f \cdot d \mu_{\mathrm{eq}}(K, h^L).
		\end{aligned}
		\end{equation}
		From (\ref{eq_non_empt}), we deduce that
		\begin{equation}\label{eq_non_empt2}
			\int_{X} f \cdot d \mu_{\mathrm{eq}}(K, h^L)
			\geq
			\epsilon.
		\end{equation}
		\par 
		We will now obtain a contradiction with our initial assumption (\ref{eq_contr}).
		We assume $k_0 \in \nat$ is big enough so that $n_k > 0$ for any $k \geq k_0$.
		We denote by $\Sigma$ the union of base loci of $L^{k_0}, L^{k_0 + 1}, \ldots, L^{2 k_0}$.
		By definition of the base-loci and the fact that the space of holomorphic sections of tensor powers of $L$ form an algebra under the multiplication, for any $x \in X \setminus \Sigma$ and $k \geq k_0$, there is $s \in H^0(X, L^{\otimes k})$ so that $s(x) \neq 0$.
		Remark also that $\Sigma$ is an analytic subset, and since $\mu$ is non-pluripolar, we have $\mu(\Sigma) = 0$.
		\par 
		Now, for any $x \in X \setminus \Sigma$, we denote by $s_{k, x} \in H^0(X, L^{\otimes k})$ (resp. $s'_{k, x} \in H^0(X, L^{\otimes k})$) the \textit{peak section} at $x$ with respect to the scalar product ${\textrm{Hilb}}_k(h^L, \mu)$ (resp. ${\textrm{Hilb}}_k(h^L, \mu')$).
		Recall that this means that $s_{k, x}$ (resp. $s'_{k, x}$) is of unit norm with respect to ${\textrm{Hilb}}_k(h^L, \mu)$ (resp. ${\textrm{Hilb}}_k(h^L, \mu')$) and orthogonal to the subspace $H^0(X, L^{\otimes k} \otimes \mathcal{J}_x)$ of holomorphic sections of $L^{\otimes k}$ vanishing at $x$.
		\par 
		Directly from the definition of the Bergman kernel and the fact that it doesn't depend on the choice of an orthonormal basis, we deduce that the Bergman kernels $B_k(x, y)$ (resp. $B'_k(x, y)$), $x \in X \setminus \Sigma, y \in X$, associated with ${\textrm{Hilb}}_k(h^L, \mu)$ (resp. ${\textrm{Hilb}}_k(h^L, \mu')$) verify
		\begin{equation}\label{eq_bk_id}
			B_k(x, y) = s_{k, x}(x) \cdot s_{k, x}(y)^*, \qquad B'_k(x, y) = s'_{k, x}(x) \cdot s'_{k, x}(y)^*.
		\end{equation}
		Immediately from the definition of peak sections, we deduce the following bound
		\begin{equation}\label{eq_peak_char}
			\frac{|s'_{k, x}(x)|^2_{(h^L)^k}}{\int_X |s'_{k, x}(y)|^2_{(h^L)^k} d \mu'(y)}
			\geq
			\frac{|s_{k, x}(x)|^2_{(h^L)^k}}{\int_X |s_{k, x}(y)|^2_{(h^L)^k} d \mu'(y)},
		\end{equation}
		which according to (\ref{eq_bk_id}) and our normalization yields
		\begin{equation}\label{eq_peak_char1}
			\int_X |s_{k, x}(y)|^2_{(h^L)^k} d \mu'(y)
			\geq
			\frac{|B_k(x, x)|}{|B'_k(x, x)|}.
		\end{equation}
		From (\ref{eq_bk_id}), we deduce
		\begin{multline}\label{eq_int_bk_peak}
			\int_{x \in X \setminus \Sigma} \int_{y \in X} f(x) \cdot |B_k(x, y)|^{2}_{(h^L)^k} \cdot d \mu'(x) \cdot d \mu'(y)
			\\
			=
			\int_{x \in X \setminus \Sigma} f(x) \cdot |B_k(x, x)| \cdot \Big( \int_{y \in X} |s_{k, x}(y)|^2_{(h^L)^k} \cdot d \mu'(y) \Big) \cdot d \mu'(x)
		\end{multline}
		By (\ref{eq_bk_id}) and (\ref{eq_peak_char1}), we further obtain 
		\begin{multline}\label{eq_int_bk_peak2}
			\int_{x \in X \setminus \Sigma} f(x) \cdot |B_k(x, x)| \cdot \Big( \int_{y \in X} |s_{k, x}(y)|^2_{(h^L)^k} d \mu'(y) \Big) \cdot d \mu'(x)
			\\
			\geq
			\int_{x \in X \setminus \Sigma} f(x) \cdot \frac{|B_k(x, x)|^2}{|B'_k(x, x)|} \cdot d \mu'(x).
		\end{multline}
		By Cauchy-Schwartz inequality and the fact that $\mu'$ coincides with $\mu$ over the support of $f$,
		\begin{multline}\label{eq_int_bk_peak3}
			\Big( \int_{x \in X \setminus \Sigma} f(x) \cdot \frac{|B_k(x, x)|^2}{|B'_k(x, x)|} \cdot d \mu'(x) \Big)
			\cdot
			\Big(
			\int_{x \in X \setminus \Sigma} f(x) \cdot |B'_k(x, x)| \cdot d \mu'(x)
			\Big)
			\\
			\geq
			\Big( \int_{x \in X \setminus \Sigma} f(x) \cdot |B_k(x, x)| \cdot d \mu(x) \Big)^2.
		\end{multline}
		A combination of (\ref{eq_asympt_bk_diag}), (\ref{eq_non_empt2}), (\ref{eq_int_bk_peak}), (\ref{eq_int_bk_peak2}), (\ref{eq_int_bk_peak3}) and $\mu(\Sigma) = 0$ yields
		\begin{equation}\label{eq_lower_bnd}
			\liminf_{k \to \infty}
			\frac{1}{n_k}
			\int_{x \in X} \int_{y \in X} f(x) \cdot |B_k(x, y)|^{2}_{(h^L)^k} \cdot d \mu'(x) \cdot d \mu'(y)
			\geq
			\int_{X} f \cdot d \mu_{\mathrm{eq}}(K, h^L).
		\end{equation}
		\par 
		Let us now estimate from above the integral on the left-hand side of (\ref{eq_lower_bnd}) to obtain a contradiction with (\ref{eq_contr}) -- a statement we have not yet used.
		Immediately from the definition of $f$ and $\mu'$,
		\begin{multline}\label{eq_upp_bnd1}
			\int_{x \in X} \int_{y \in X} f(x) \cdot |B_k(x, y)|^{2}_{(h^L)^k} \cdot d \mu'(x) \cdot d \mu'(y)
			\leq
			\\
			\int_{x \in X} \int_{y \in X} f(x) \cdot |B_k(x, y)|^{2}_{(h^L)^k} \cdot d \mu(x) \cdot d \mu(y)
			\\
			-
			(1 - e^{-1}) \cdot \int_{x \in K_1} \int_{y \in K_2} |B_k(x, y)|^{2}_{(h^L)^k} \cdot d \mu(x) \cdot d \mu(y).
		\end{multline}
		From (\ref{eq_push_forw}), we deduce
		\begin{equation}\label{eq_upp_bnd2}
			\int_{x \in X} \int_{y \in X} f(x) \cdot |B_k(x, y)|^{2}_{(h^L)^k} \cdot d \mu(x) \cdot d \mu(y)
			=
			\int_{x \in X} f(x) \cdot |B_k(x, x)| \cdot d \mu(x).
		\end{equation}
		From (\ref{eq_contr}), (\ref{eq_asympt_bk_diag}), (\ref{eq_upp_bnd1}) and  (\ref{eq_upp_bnd2}), we obtain
		\begin{multline}
			\limsup_{k \to \infty}
			\frac{1}{n_k}
			\int_{x \in X} \int_{y \in X} f(x) \cdot |B_k(x, y)|^{2}_{(h^L)^k} \cdot d \mu'(x) \cdot d \mu'(y)
			\\
			\leq
			\int_{X} f \cdot d \mu_{\mathrm{eq}}(K, h^L) - (1 - e^{-1}) \epsilon,
		\end{multline}
		clearly contradicting (\ref{eq_lower_bnd}), and, hence, our initial assumption (\ref{eq_contr}).
		This finishes our proof.
	\end{proof}
	
	\begin{proof}[Proof of Theorem \ref{thm_off_diag}]
		Let us fix $f \in \ccal^0(X \times X)$.
		We would like to verify that $\int f(x, y) \cdot d \mu_k^{\rm{Berg}}(x, y) \to \int f(x, x) \cdot d \mu_{eq}(K, h^L)$, as $k \to \infty$.
		We define $g \in \ccal^0(X \times X)$ as $g(x, y) := f(x, x)$.
		Then directly from Theorem \ref{thm_diag} and (\ref{eq_push_forw}), we obtain $\int g(x, y) \cdot d \mu_k^{\rm{Berg}}(x, y) \to \int f(x, x) \cdot d \mu_{eq}(K, h^L)$, as $k \to \infty$.
		By considering the difference $f - g$, we see that it suffices to show that $\int h(x, y) \cdot d \mu_k^{\rm{Berg}}(x, y) \to 0$, as $k \to \infty$, for continuous $h$ vanishing on the diagonal.
		\par 
		According to Theorem \ref{thm_nomass}, the above holds for $h$ lying in the space of the functions $\mathcal{V}$ spanned by $a(x) \cdot b(y)$ where $a, b \in \ccal^0(X)$ have non-intersecting support.
		Hence it also holds for the functions from the uniform closure $\overline{\mathcal{V}}$ of $\mathcal{V}$.
		The proof of Theorem \ref{thm_off_diag} will be complete once we establish that $\overline{\mathcal{V}}$ coincides with the space of functions vanishing on the diagonal.
		\par 
		To establish this, note first that it is immediate that every function from $\overline{\mathcal{V}}$ vanishes on the diagonal.
		Also any function vanishing along the diagonal can be uniformly approximated by functions vanishing in a neighborhood of the diagonal.
		It is hence enough to show that a continuous function $h$ vanishing in a neighborhood of the diagonal lies in $\overline{\mathcal{V}}$.
		To see this, consider a partition of unity $\rho_i$, $i \in I$, subordinate to a sufficiently small mesh.
		Then from the uniform continuity of $h$, one sees that the functions $\sum_{i, j \in I} h(x_{i, j}) \rho_i(x) \rho_j(y)$, where $x_{i, j} \in {\rm{supp}} (\rho_i) \times {\rm{supp}} (\rho_j)$ are chosen in an arbitrary way, approximate uniformly the function $h$ if the size of the mesh is small enough, and -- again if the mesh is small enough -- these approximations lie in $\mathcal{V}$ by our assumption on the vanishing of $h$ in a neighborhood of the diagonal.
	\end{proof}
	
	\section{Algebraic and spectral aspects of Toeplitz operators}\label{sect_toepl}
	The main goal of this section is to study the space of Toeplitz operators and to establish Theorems \ref{thm_alg} and \ref{thm_distr}.
	This will be achieved as a consequence of the following result, the proof of which is based on Theorem \ref{thm_off_diag}.
	\begin{thm}\label{thm_prod_thm}
		For any $f, g \in \ccal^0(X)$, $\{ T_k(f) \circ T_k(g), k \in \nat \}$ forms a Toeplitz operator with symbol $f \cdot g$.
	\end{thm}
	We shall use the following result concerning the spectral radius.
	\begin{sloppypar}
	\begin{lem}\label{lem_spec_radius}
		For any $f \in \ccal^0(X)$, the minimal and the maximal eigenvalues $\lambda_{\min}(T_k(f))$, $\lambda_{\max}(T_k(f))$ of $T_k(f)$ satisfy $\lambda_{\min}(T_k(f)) \geq \inf_{x \in K} f(x)$ and $\lambda_{\max}(T_k(f)) \leq \sup_{x \in K} f(x)$.
	\end{lem}
	\begin{proof}
		Replacing $f$ with $-f$, the problem reduces to analyzing $\lambda_{\max}(T_k(f))$.
		The bound then follows immediately from the min-max characterization of the eigenvalues and the bound $\langle T_k(f) s, s \rangle_{{\textrm{Hilb}}_k(h^L, \mu)} \leq \sup_{x \in K} f(x) \cdot \langle s, s \rangle_{{\textrm{Hilb}}_k(h^L, \mu)}$ for any $s \in H^0(X, L^{\otimes k})$, which follows immediately from the definition of the $L^2$-norm.
	\end{proof}
	\end{sloppypar}
	\begin{proof}[Proof of Theorem \ref{thm_prod_thm}]
		We remark that the uniform bound on the operator norm of $T_k(f) \circ T_k(g)$ is an immediate consequence of Lemma \ref{lem_spec_radius}.
		Classical properties of Schatten norms yield -- in the notations of Definition \ref{defn_toepl_sch} -- that for any $T, S \in {\enmr{H^0(X, L^{\otimes k})}}$, $k \in \nat$, we have
		\begin{equation}\label{eq_schatted_bnds}
			\| T \|_p \leq \| T \|_2^{\frac{1}{p}} \cdot  \| T \|^{\frac{p - 1}{p}}, \qquad \| S \circ T \|_p \leq \| S \| \cdot \| T \|_p.
		\end{equation}
		So it suffices to establish that 
		\begin{equation}\label{eq_sk_schw00}
			\lim_{k \to \infty}
			\| T_k(f) \circ T_k(g) - T_k(f \cdot g) \|_2 
			=
			0.
		\end{equation}
		Now, it is immediate to see using the reproducing property $B_k \circ B_k = B_k$ that we can write
		\begin{equation}\label{eq_sk_schw0}
			T_k(f) \circ T_k(g) - T_k(f \cdot g) 
			=
			B_k \circ S_k \circ B_k,
		\end{equation}
		where $B_k$ is the Bergman kernel as in Definition \ref{defn_toepl_sch}, and $S_k \in {\enmr{H^0(X, L^{\otimes k})}}$, is defined as 
		\begin{equation}
			(S_k s)(x) := \int_{y \in X} S_k(x, y) \cdot s(y) \cdot d \mu(y), \quad \text{for any } s \in H^0(X, L^{\otimes k}),
		\end{equation}
		for $S_k(x, y) \in L^{\otimes k}_x \otimes (L^{\otimes k}_y)^*$ given by
		\begin{equation}\label{eq_sk_schw}
			S_k(x, y) = (f(x) g(x) - f(x) g(y)) B_k(x, y).
		\end{equation}
		From (\ref{eq_schatted_bnds}), we obtain
		\begin{equation}\label{eq_sk_schw1}
			\| B_k \circ S_k \circ B_k \|_2 \leq \| S_k \|_2.
		\end{equation}
		We now rely on the fact that the $2$-Schatten norm is the rescaled Hilbert-Schmidt norm, and the latter can be calculated using the $L^2$-norm of Schwartz kernel of the operator, which gives us 
		\begin{equation}\label{eq_sk_schw2}
			\| S_k \|_2^2
			=
			\frac{1}{n_k} \int_{X \times X} |S_k(x, y)|^{2}_{(h^L)^k} \cdot d \mu(x) \cdot d \mu(y).
		\end{equation}
		But from Theorem \ref{thm_off_diag} and (\ref{eq_sk_schw}), we deduce that 
		\begin{equation}
			\lim_{k \to \infty} \frac{1}{n_k} \int_{X \times X} |S_k(x, y)|^{2}_{(h^L)^k} \cdot d \mu(x) \cdot d \mu(y)
			=
			0,
		\end{equation}
		which along with (\ref{eq_sk_schw0}), (\ref{eq_sk_schw}), (\ref{eq_sk_schw1}) and (\ref{eq_sk_schw2}) finishes the proof of (\ref{eq_sk_schw00}).
	\end{proof}
	\begin{cor}\label{cor_prod_toepl}
		Consider two sequences of operators $T_k, T'_k \in {\enmr{H^0(X, L^{\otimes k})}}$, $k \in \nat$, forming Toeplitz operators with symbols $f \in \ccal^0(X)$ and $g \in \ccal^0(X)$.
		Then the sequence of operators $T_k \circ T'_k \in {\enmr{H^0(X, L^{\otimes k})}}$, $k \in \nat$, forms a Toeplitz operator with a symbol $f \cdot g \in \ccal^0(X)$.
	\end{cor}
	\begin{proof}
		The bound $\| T_k \circ T'_k \| \leq \| T_k \| \cdot \| T'_k \|$ implies that $\| T_k \circ T'_k \|$ is uniformly bounded.
		By the triangle inequality, we have
		\begin{multline}
			\big\| T_k \circ T'_k - T_k(f \cdot g) \big\|_p
			\leq
			\big\| (T_k - T_k(f)) \circ T'_k  \big\|_p
			\\
			+
			\big\| T_k(f) \circ (T'_k - T_k(g)) \big\|_p
			+
			\big\| T_k(f) \circ T_k(g) - T_k(f \cdot g) \big\|_p,
		\end{multline}
		which finishes the proof of Corollary \ref{cor_prod_toepl} by Theorem \ref{thm_prod_thm} and (\ref{eq_schatted_bnds}).
	\end{proof}		
	
	Let us now show that the space of Toeplitz operators is closed under continuous functional calculus.
	\begin{cor}\label{cor_func_calc}
		For any $h \in \ccal^0(\real)$ and a sequence $\{ T_k , k \in \nat \}$, forming a Toeplitz operator with symbol $f \in \ccal^0(X)$, the sequence $\{ h(T_k), k \in \nat \}$, forms a Toeplitz operator with symbol $h(f)$.
	\end{cor}
	\begin{proof}
		Remark that the uniform bound on the operator norm of $h(T_k)$ is immediate.
		Now, since $T_k$ has a uniformly bounded spectrum, let's say located inside of a compact interval $I$, the operator $h(T_k)$ depends solely on the restriction of $h$ on $I$.
		By Weierstrass approximation theorem, applied over $I$, we see that it suffices to establish the result for $h$ given by monomials.
		But this follows immediately from Corollary \ref{cor_prod_toepl} and finishes the proof.
		\par 
		Note also that from Lemma \ref{lem_spec_radius} and the above proof, we see that the analogue of Corollary \ref{cor_func_calc} holds if we replace $T_k$ by $T_k(f)$ and $h$ by a continuous function on $[\min f, \max f]$.
	\end{proof}
	We will need the following lemma, which -- although a special case of Theorem \ref{thm_distr} -- is actually used in its proof.
	\begin{lem}\label{lem_trace_toepl}
		For any $f \in \ccal^0(X)$, we have $\lim_{k \to \infty} \frac{1}{n_k} {\rm{Tr}}[T_k(f)] = \int f \cdot \mu_{\mathrm{eq}}(K, h^L)$.
	\end{lem}
	\begin{proof}
		Immediately from the fact that trace can be calculated through the integral of the Schwartz kernel, we deduce
		\begin{equation}
			{\rm{Tr}}[T_k(f)]
			=
			\int_{X \times X} f(x) \cdot |B_k(x, y)|^{2}_{(h^L)^k} \cdot d \mu(x) \cdot d \mu(y).
		\end{equation}
		The result now follows from Theorem \ref{thm_distr} and (\ref{eq_push_forw}).
	\end{proof}
	We can now finally establish that the symbol map is well-defined.
	\begin{prop}\label{prop_symb_well_def}
		Let $f, g \in \ccal^0(X)$.
		Then $\{ T_k(f) , k \in \nat \}$, forms a Toeplitz operator with symbol $g$, if and only if $f = g$ over $K'$, where $K'$ was defined in (\ref{eq_kprim}).
	\end{prop}
	\begin{proof}
		According to Corollary \ref{cor_func_calc}, the sequence of operators $\{ |T_k(f) - T_k(g)| , k \in \nat \}$, forms a Toeplitz operator with symbol $|f - g|$.
		By Lemma \ref{lem_trace_toepl}, we conclude that
		\begin{equation}\label{eq_symb_wd_2}
			\lim_{k \to \infty} \frac{1}{n_k} {\rm{Tr}}[|T_k(f) - T_k(g)|] = \int |f - g| \cdot \mu_{\mathrm{eq}}(K, h^L).
		\end{equation}
		According to (\ref{eq_schatted_bnds}), $\{ T_k(f) , k \in \nat \}$, forms a Toeplitz operator with symbol $g$ if and only if
		\begin{equation}\label{eq_symb_wd_1}
			\lim_{k \to \infty} \frac{1}{n_k} {\rm{Tr}} \Big[ |T_k(f) - T_k(g)| \Big] = 0.
		\end{equation}
		A comparison between (\ref{eq_symb_wd_2}) and (\ref{eq_symb_wd_1}) shows -- by the continuity of $f$ and $g$ -- that $f = g$ on $K'$ if and only if $\{ T_k(f) , k \in \nat \}$, forms a Toeplitz operator with symbol $g$.
	\end{proof}
	\begin{proof}[Proof of Theorem \ref{thm_alg}]
		It follows immediately from Corollary \ref{cor_prod_toepl} and Proposition \ref{prop_symb_well_def}.
	\end{proof}
	\begin{proof}[Proof of Theorem \ref{thm_distr}]
		Remark first that 
		\begin{equation}
			\sum_{\lambda \in {\rm{Spec}} (T_k)} g(\lambda)
			=
			{\rm{Tr}}[g(T_k)].
		\end{equation}
		The result now follows immediately from Corollary \ref{cor_func_calc} and Lemma \ref{lem_trace_toepl}.
		\par 
		Note also that from Lemma \ref{lem_spec_radius} and the above proof, we see that the analogue of Theorem \ref{thm_distr} holds if we replace $T_k$ by $T_k(f)$ and $h$ by a continuous function on $[\min f, \max f]$.
	\end{proof}
	
	We finish this section with the following proposition concerning the spectral radius of the Toeplitz operators, which should be compared with Lemma \ref{lem_spec_radius}.
	\begin{prop}\label{prop_spec_radius}
		In the notations of Lemma \ref{lem_spec_radius}, we have
		\begin{equation}
			\limsup_{k \to \infty} \lambda_{\max}(T_k(f)) \geq \sup_{x \in K'} f(x),
			\qquad
			\liminf_{k \to \infty} \lambda_{\min}(T_k(f)) \leq \inf_{x \in K'} f(x),
		\end{equation}
		where $K'$ was defined in (\ref{eq_kprim}).
	\end{prop}
	\begin{sloppypar}
	\begin{proof}
		Replacing $f$ with $-f$, the problem reduces to analyzing $\lambda_{\max}(T_k(f))$.
		We let $r := \sup_{x \in K'} f(x)$, and for a given $\epsilon > 0$, we consider a continuous function $g: \real \to [0, 1]$, so that $g(r) = 1$ and $g(y) = 0$ for $y < r - \epsilon$.
		Then by our assumption we have $\int_X g(f(x)) d \mu_{eq}(K, h^L)(x) > 0$. 
		Proposition \ref{prop_spec_radius} follows from this and Theorem \ref{thm_distr}.
	\end{proof}
	\end{sloppypar}
	
	\section{Orthogonal polynomials and Toeplitz matrices}\label{sect_mat}
	The main goal of this section is to show how the general theory developed above specializes to the classical study of Toeplitz matrices and orthogonal polynomials. 
	\par 
	We consider a compact subset $K \subset \comp^n$ as a subset of $\mathbb{P}^n$ through the standard affine chart, and endow $L := \mathscr{O}(1)$ with the (singular at $\{ \infty \}$) metric $h^L$ so that for a global holomorphic section $\sigma$ of $\mathscr{O}(1)$ having a zero at $\{ \infty \}$, we have $|\sigma(x)|_{h^L} = 1$. 
	Then the Bernstein-Markov condition on a positive measure $\mu$ supported on $K$ associated with $(K, h^L)$ -- note that the restriction of $h^L$ to $K$ is continuous and so we can speak of the Bernstein-Markov condition on $(K, h^L)$ -- can be rewritten in the following way: for any $\epsilon > 0$, there is $C > 0$ so that for any polynomial $P$, we have
	\begin{equation}\label{eq_bm_clas0}
		\sup_{x \in K} |P(x)|^2
		\leq
		C
		\cdot
		\exp(\epsilon \deg(P))
		\cdot
		\int |P(x)|^2 d \mu(x).
	\end{equation}
	The psh envelope $h^L_K$ can be expressed as $h^L_K = h^L \cdot \exp(-\phi_K)$, where the function $\phi_K$ is given by
	\begin{equation}
		\phi_K(z) := \sup \{ v(z) : v \in \mathcal{L}, v \leq 0 \text{ on } K \},
	\end{equation}
	and $\mathcal{L}$ is the Lelong class of functions, consisting of psh functions on $\comp^n$ so that up to a constant, they are bounded above by $\frac{1}{2} \log (1 + |z|^2)$, $z \in \comp^n$, cf. \cite[Example 1.2]{GuedZeriGeomAnal}.
	\par 
	The main results of this paper yield asymptotic statements from the theory of orthogonal polynomials.  
	Rather than presenting them in their most general form, we prefer to illustrate two important special cases.
	\par 
	Let $\mathbb{S}^1$ be the unit circle and $f \in \ccal^0(\mathbb{S}^1)$, $f \neq 0$, be a fixed real function, which write as $f(\theta) = \sum_{i = - \infty}^{+ \infty} a_j \exp(\imun j \theta)$, where $\theta \in [0, 2 \pi[$, gives a standard parametrization of $\mathbb{S}^1$. 
	Then $a_i = \overline{a}_{-i}$ and not all $a_i$ vanish.
	\par 
	Now, we embed $\mathbb{S}^1$ in $X := \mathbb{P}^1$ as one of the great circles (for concreteness given by $\theta \mapsto [1: \exp(i \theta)] \in \mathbb{P}^1$, $\theta \in [0, 2 \pi[$, where $[1 : z] \in \mathbb{P}^1$, $z \in \comp$, is a standard affine chart) and denote by $\mu$ the Lebesgue measure on $\mathbb{S}^1$, viewed as a measure on $X$.
	Remark that as $\mathbb{S}^1$ is totally real, the measure $\mu$ is non-pluripolar, see \cite{SadullaevReal}, cf. \cite[Exercise 4.39.5]{GuedjZeriahBook}. 
	\par 
	We then take $L := \mathscr{O}(1)$ and endow it with the metric described above.
	The measure $\mu$ is Bernstein-Markov with respect to $(\mathbb{S}^1, h^L)$ by \cite[Corollary 1.8]{BerBoucNys}, \cite[Proposition 3.6]{BernsteinMarkovSurvey}.
	\par 
	Remark that in a standard basis of monomials of $H^0(X, L^{\otimes k})$ (given by $z^i$, $i = 0, \ldots, k$, in the already mentioned affine chart), the operator $T_k(f)$ writes as the following Toeplitz matrix 
	\begin{equation}
		T_k[f]
		:=
		\begin{bmatrix}
		a_0 & a_{-1} & a_{-2} & \cdots & a_{-k + 1} \\
		a_1 & a_0 & a_{-1} & \cdots & a_{-k+2} \\
		a_2 & a_1 & a_0 & \cdots & a_{-k+3} \\
		\vdots & \vdots & \vdots & \ddots & \vdots \\
		a_{k - 1} & a_{k-2} & a_{k-3} & \cdots & a_0
		\end{bmatrix}.
	\end{equation}
	\par 
	Due to $\mathbb{S}^1$-symmetry, the equilibrium measure associated with $(\mathbb{S}^1, h^L)$ is just the Lebesgue measure on the circle, cf. \cite[Exercise 9.8]{GuedjZeriahBook}.
	Immediately from this and Theorem \ref{thm_distr}, we recover the following classical statement.
	\begin{thm}[{Szegő first limit theorem \cite{SzegoFST}}]
		For any continuous $f : \mathbb{S}^1 \to \real$, $g: \real \to \real$,
		\begin{equation}
			\lim_{k \to \infty} \frac{1}{k} \sum_{\lambda \in {\rm{Spec}} (T_k[f])} g(\lambda)
			=
			\frac{1}{2 \pi}
			\int_0^{2 \pi} g(f(\theta)) d \theta.
		\end{equation}
	\end{thm}
	\begin{rem}
		The statement holds under much laxer assumptions on $f$, see \cite{NikolskiBook} for details.
	\end{rem}
	\par 
	Theorem \ref{thm_alg} then immediately recovers the following statement.
	\begin{thm}[{Grenander-Szeg{\"o} \cite[\S 7, 8]{GrenanSzego}}]
		For any continuous $f, g : \mathbb{S}^1 \to \real$, we have
		\begin{equation}
			\lim_{k \to \infty} \frac{1}{k} {\rm{Tr}} \Big[ \Big| T_k[f] \circ T_k[g] - T_k[f \cdot g] \Big| \Big]
			=
			0.
		\end{equation}
	\end{thm}
	\par 
	Let us now describe another application.
	For this, in the considered above affine chart of $\mathbb{P}^1$, we consider a subset $K := [-1, 1] \subset \real \subset \comp$ and let $\mu$ be the Lebesgue measure on $K$.
	\par 
	Then the measure $\mu$ is Bernstein-Markov with respect to $(K, h^L)$ by \cite[Corollary 1.8]{BerBoucNys}, \cite[Proposition 3.6]{BernsteinMarkovSurvey}.
	The classical calculation due to Lundin \cite{Lundin}, cf. \cite[p. 707]{BedfordTaylorEquil}, shows that the equilibrium measure is given by
	\begin{equation}
		\mu_{\mathrm{eq}}(K, h^L) = \frac{dx|_{[-1, 1]}}{\pi \cdot \sqrt{1 - x^2}}.
	\end{equation}
	\par 
	Now, let us denote by $L_n(z)$, $z \in \comp$, $n \in \nat$, the normalized Legendre polynomials.
	Recall that this means that $L_n(z)$ has degree $n$, and the following orthogonality relation holds
	\begin{equation}\label{eq_orth_poly}
		\int_{-1}^{1} L_n(x) L_m(x) dx = 1 \cdot \delta_{n m},
	\end{equation}
	where $\delta_{n m}$ is the Kronecker symbol.
	\par 
	Using the above trivialization of the bundle $L$, we see that for any $k \in \nat$, the polynomials $\{ L_n(z), n = 0, 1, \ldots, k - 1 \}$ generate $H^0(X, L^{\otimes k})$.
	Moreover, the condition (\ref{eq_orth_poly}) means exactly that these polynomials form an orthogonal basis with respect to ${\textrm{Hilb}}_k(h^L, \mu)$.
	\par 
	Let us now consider $f \in \ccal^0([-1, 1])$ and the associated Toeplitz operators $T_k(f)$.
	In the above orthonormal basis the operator $T_k(f)$ writes as the following Toeplitz matrix 
	\begin{equation}
		T_k\{ f \}
		:=
		\begin{bmatrix}
		a_{0 0} & a_{0 1} & \cdots & a_{0 k - 1} \\
		a_{1 0} & a_{1 1} & \cdots & a_{1 k - 1} \\
		\vdots & \vdots & \ddots & \vdots \\
		a_{k - 1 0} & a_{k-1 1} & \cdots & a_{k-1 k-1}
		\end{bmatrix},
	\end{equation}
	where $a_{i j} := \int_{-1}^{1} f(x) \cdot L_i(x) L_j(x) dx$, $i, j \in \nat$.
	Immediately from this and Theorem \ref{thm_distr}, we recover the following result.
	\begin{thm}[{Grenander-Szeg{\"o} \cite[p. 116]{GrenanSzego} and Nevai \cite[\S 5]{Nevai}}]
		For any continuous $f : [-1, 1] \to \real$, $g: \real \to \real$, we have
		\begin{equation}
			\lim_{k \to \infty} \frac{1}{k} \sum_{\lambda \in {\rm{Spec}} (T_k\{ f \})} g(\lambda)
			=
			\frac{1}{\pi}
			\int_{-1}^{1} \frac{g(f(x)) dx}{\sqrt{1 - x^2}}.
		\end{equation}
	\end{thm}
	\par 
	Theorem \ref{thm_alg} then immediately gives us the following statement, which resembles some results from {Grenander-Szeg{\"o} \cite[\S 8.1]{GrenanSzego}}.
	\begin{thm}
		For any continuous $f, g : [-1, 1] \to \real$, we have
		\begin{equation}
			\lim_{k \to \infty} \frac{1}{k} {\rm{Tr}} \Big[ \Big| T_k\{ f \} \circ T_k\{ g \} - T_k \{ f \cdot g \} \Big| \Big]
			=
			0.
		\end{equation}
	\end{thm}
	\par 
	To conclude, we briefly recall that for symmetric convex subsets $K \subset \mathbb{R}^n$, an explicit formula for the equilibrium measure associated with $(K, h^L)$ was obtained by Bedford-Taylor \cite{BedfordTaylorEquil}. 
	According to Lundin's formulas for $h^L_K$ from \cite{Lundin}, see also \cite[p. 707]{BedfordTaylorEquil}, the pair $(K, h^L)$ is pluriregular. 
	It then follows from \cite[Proposition 1.13 and Theorem 1.14]{BerBoucNys} that the corresponding equilibrium measures satisfy the Bernstein-Markov property.  
	Thus, the analogues of the above result can be formulated for arbitrary symmetric convex subsets $K \subset \mathbb{R}^n$.
	We leave the precise formulation to the interested reader.

	\section{Bergman measures on singular spaces}\label{sect_sing}
	We fix from now on a compact complex reduced irreducible analytic space $X$, $\dim_{\comp} X = n$, and a big line bundle $L$ over $X$.
	We fix a continuous Hermitian metric $h^L$ on $L$ and a positive \textit{volume form} $\mu := dV$ on $X$ (which means it is a pull-back of a (strictly) positive $(n, n)$-differential form in each chart).
	For the background on the singular complex spaces, we recommend \cite[\S II.9]{DemCompl}
	We denote by $B_k(x, y)$ the associated Bergman kernel, defined as in (\ref{eq_bergm_kern}), and $\mu_k^{\rm{Berg}}$ as in (\ref{eq_mu_begmm}).
	\par 
	The asymptotics of $\frac{1}{k} \log |B_k(x, x)|$ has been investigated in a series of recent papers by Coman-Marinescu \cite{ComMar}, and by Coman-Ma-Marinescu \cite{ComanMaMarinGT}, \cite{ComanMaMarMoish}, see also Berman \cite[Theorem 1.5]{BermanEnvProj} and Bayraktar \cite[Proposition 2.9]{BayraktarEquidistr} for related results in this direction.  
	Our main focus in this section is on the non-logarithmic asymptotics, both on and off the diagonal.
	\par 
	In order to state our results, we define first the equilibrium measure associated with $(X, h^L)$.  
	To this end, let $\pi : \hat{X} \to X$ be a resolution of singularities.  
	According to \cite[Proposition 2.2.43]{LazarBookI}, the pullback line bundle $\pi^* L$ is big and
	\begin{equation}\label{eq_vol_resol}
    	\mathrm{vol}(\pi^* L) = \mathrm{vol}(L).
	\end{equation}
	We then define the equilibrium measure $\mu_{\mathrm{eq}}(X, h^L)$ on $X$ as
	\begin{equation}\label{eq_defn_eq_sing}
    	\mu_{\mathrm{eq}}(X, h^L) := \pi_* \mu_{\mathrm{eq}}(\hat{X}, \pi^* h^L).
	\end{equation}
	Note that, since any two resolutions of singularities can be dominated by a third one, the above definition is independent of the choice of resolution by \cite[Proposition 1.9]{BermanBouckBalls}.
	\par 
	For $k \in \nat$, we denote by $B_k(x, y)$ (resp. $\hat{B}_k(x, y)$) the Bergman kernel associated with $X$, $L$, $h^L$ and $\mu$ (resp. $\hat{X}$, $\pi^* L$, $\pi^* h^L$ and $\pi^* \mu$), defined as in (\ref{eq_bergm_kern}), and by $\mu_k^{\rm{Berg}}$ (resp. $\hat{\mu}_k^{\rm{Berg}}$) the measure on $X \times X$ (resp. $\hat{X} \times \hat{X}$), defined as in (\ref{eq_mu_begmm}).
	We now state the two main results of this section.
	\par 
	\begin{thm}\label{thm_sing0}
		The sequence of measures $\frac{1}{n_k} |B_k(x, x)| \cdot d \mu(x)$ on $X$ converges weakly, as $k \to \infty$, to the equilibrium measure $\mu_{eq}(X, h^L)$.
	\end{thm}
	\begin{thm}\label{thm_sing}
		If the space $X$ is normal, then the sequence of measures $\mu_k^{\rm{Berg}}$ on $X$ converges weakly, as $k \to \infty$, to $\Delta_* \mu_{eq}(X, h^L)$, where $\Delta: X \to X \times X$ is the diagonal embedding.
		Moreover, the analogues of Theorems \ref{thm_alg} and \ref{thm_distr} hold in this setting.
	\end{thm}
	Let us begin by establishing some preliminary results.
	For brevity, we denote 
	\begin{equation}
		\mu_k^B := \frac{1}{n_k} |B_k(x, x)| \cdot d \mu(x), \qquad \hat{\mu}_k^B := \frac{1}{\hat{n}_k} |\hat{B}_k(x, x)| \cdot d \hat{\mu}(x),
	\end{equation}
	where $\hat{n}_k := \dim H^0(\hat{X}, \pi^* L^{\otimes k})$ and $n_k := \dim H^0(X, L^{\otimes k})$.
	\begin{lem}\label{lem_comp_res_meas}
		As $k \to \infty$, the sequence of measures $\pi_* \hat{\mu}_k^B - \mu_k^B$ converges weakly to zero.
	\end{lem}
	\begin{proof}
		According to (\ref{eq_vol_resol}), $n_k \sim \hat{n}_k$, as $k \to \infty$, and so it suffices to establish that the measures
		\begin{equation}\label{eq_meas_aux}
			\frac{1}{n_k} \Big( |\hat{B}_k(x, x)| - |B_k(\pi(x), \pi(x))| \Big) \cdot d \mu(x),
		\end{equation}
		converge weakly to zero, as $k \to \infty$.
		From the definitions, we have 
		\begin{equation}
			\int_{\hat{X}}  \Big( |\hat{B}_k(x, x)| - |B_k(\pi(x), \pi(x))| \Big) \cdot \pi^* d \mu(x)
			=
			\dim H^0(\hat{X}, \pi^* L^{\otimes k})
			-
			\dim H^0(X, L^{\otimes k}).
		\end{equation}
		Since by (\ref{eq_vol_resol}), the right-hand side of the above equality is asymptotically negligible compared to $n_k$, we deduce that the mass of the sequence of measures (\ref{eq_meas_aux}) tends to zero as $k \to \infty$.  
		However, the embedding of $H^0(X, L^{\otimes k})$ into $H^0(\hat{X}, \pi^* L^{\otimes k})$ is isometric with respect to the associated $L^2$-products, so the difference $|\hat{B}_k(x, x)| - |B_k(\pi(x), \pi(x))|$ is positive as it can be written as $\sum |s_i(x)|^2_{(h^L)^2}$ where $s_i$ form a basis of the orthogonal complement to the image of $H^0(X, L^{\otimes k})$ in $H^0(\hat{X}, \pi^* L^{\otimes k})$. 
		Hence the measures in (\ref{eq_meas_aux}) are positive, and as their mass tends to zero, the measures themselves tend weakly to zero as $k \to \infty$, which completes the proof.
	\end{proof}
	\begin{lem}\label{lem_bm}
		The measure $\pi^* \mu$ on $\hat{X}$ is Bernstein-Markov with respect to $(X, \pi^* h^L)$.
	\end{lem}
	In order to establish Lemma \ref{lem_bm}, we need to recall some notions from pluripotential theory.
	We thus fix a complex manifold $Z$ with a big line bundle $F$ endowed with a continuous metric $h^F$.
	We fix a non-pluripolar measure $\nu$ with support $Z$ (all the definitions and results below work for measures supported on arbitrary compact subsets, but since we shall only work with fully supported measures, we simplify our presentation accordingly).
	Following \cite{SiciakDetermMeas} and \cite{BerBoucNys}, recall that $\nu$ is called \textit{Bernstein-Markov with respect to psh weights on $(Z, h^F)$} if for any $\epsilon > 0$, there is $C > 0$, such that for any $\psi: X \to [- \infty, +\infty[$, so that $h^F \cdot \exp(- \psi)$ has a psh potential, for any $p \geq 1$, we have
	\begin{equation}\label{bm_psh_weights}
		\sup_Z \big( \exp(p \cdot \psi) \big)
		\leq
		C
		\cdot
		\exp(\epsilon p)
		\cdot
		\int_Z \exp(p \cdot \psi) d \nu.
	\end{equation}
	\par 
	It is immediate that $\nu$ is Bernstein-Markov for $(Z, h^F)$ if it is Bernstein-Markov with respect to psh weights.
	Indeed, by substituting $\psi := \frac{1}{k} \log |s|_{(h^F)^k}$ and $p = 2k$ into (\ref{bm_psh_weights}), one obtains the required bound for (\ref{eq_bm_clas}).
	\par 
	Following \cite{BermanBouckBalls}, we say that $\nu$ is \textit{determining} for $(Z, h^F)$ if for each measurable subset $E \subset Z$, $\nu(E) = 0$, we have $h^F_{Z} = h^F_{Z \setminus E}$.
	This definition differs slightly from \cite{SiciakDetermMeas}. 
	For a comparison, we refer to \cite[Proposition 3.6]{FinEigToepl}.
	\par 
	Let us recall the following result from \cite[Theorem 1.14]{BerBoucNys} (see also \cite[Theorems 5.1 and 5.2]{SiciakDetermMeas}).
	\begin{thm}\label{thm_bbwn_determining}
		The following statements are equivalent: 
		\begin{enumerate}[a)]
			\item The measure $\nu$ is determining for $(Z, h^F)$.
			\item The measure $\nu$ is Bernstein-Markov with respect to psh weights for $(Z, h^F)$.
		\end{enumerate}
	\end{thm}
	\begin{rem}\label{rem_equiv_class}
		Note that the definition of a determining measure depends only on the absolute continuity class of the measure. By Theorem \ref{thm_bbwn_determining}, the same holds true for the Bernstein-Markov property with respect to psh weights.
	\end{rem}
	\begin{proof}[Proof of Lemma \ref{lem_bm}]
		We will in fact establish a stronger statement by showing that the measure $\pi^* \mu$ on $\hat{X}$ is Bernstein-Markov with respect to psh weights on $(\hat{X}, \pi^* h^L)$.  
		First, observe that by the mean-value inequality, any positive volume form on $\hat{X}$ is Bernstein-Markov with respect to psh weights for any line bundle endowed with a continuous metric.  
		For a proof, the reader may consult \cite[Lemma 2.2]{BermanBouckBalls}, which treats the classical Bernstein-Markov property; the same argument applies verbatim in the plurisubharmonic setting.  
		Note also that the measure $\pi^* \mu$ lies in the same absolute continuity class as any positive volume form on $\hat{X}$.  
		We conclude by Remark \ref{rem_equiv_class}.
	\end{proof}
	\begin{proof}[Proof of Theorem \ref{thm_sing0}]
		By Theorem \ref{thm_diag} and Lemma \ref{lem_bm}, as $k \to \infty$, the sequence of measures $\hat{\mu}_k^B$ converges to $\mu_{\mathrm{eq}}(\hat{X}, \pi^* h^L)$.
		We conclude by this, Lemma \ref{lem_comp_res_meas}, (\ref{eq_defn_eq_sing}) and the fact that pushforwards under continuous maps preserve weak convergence.
	\end{proof}
	\begin{proof}[Proof of Theorem \ref{thm_sing}]
		Observe that Theorem \ref{thm_distr} is a formal consequence of Theorem \ref{thm_alg}, which itself follows from Theorem \ref{thm_off_diag}. 
		Thus, it is enough to prove the analogue of the latter result, on which we concentrate from now on.
		\par 
		Due to the normality of our space, the embedding of $H^0(X, L^{\otimes k})$ into $H^0(\hat{X}, \pi^* L^{\otimes k})$ is an isomorphism by Zariski's main theorem. 
		Since it is also an isometry with respect to the associated $L^2$-products, we deduce that $(\pi, \pi)_* \hat{\mu}_k^{\rm{Berg}} = \mu_k^{\rm{Berg}}$.
		However, by Theorem \ref{thm_off_diag} and Lemma \ref{lem_bm}, $\hat{\mu}_k^{\rm{Berg}}$ converge weakly, as $k \to \infty$, towards $\hat{\Delta}_* \mu_{\mathrm{eq}}(\hat{X}, \pi^* h^L)$, where $\hat{\Delta} : \hat{X} \to \hat{X} \times \hat{X}$ is the diagonal embedding.
		We conclude by this, (\ref{eq_defn_eq_sing}) and the fact that pushforwards under continuous maps preserve weak convergence.
	\end{proof}

\bibliography{bibliography}

		\bibliographystyle{abbrv}

\Addresses

\end{document}